\documentclass[11 pt]{amsart}
\usepackage{amssymb,amsmath,amsthm}
\usepackage{xcolor}
\usepackage[shortlabels]{enumitem}
\usepackage{}
\newtheorem{theorem}{Theorem}[section]

\newtheorem{lemma}[theorem]{Lemma}
\newtheorem{proposition}[theorem]{Proposition}
\newtheorem{definition}[theorem]{Definition}
\newtheorem{corollary}[theorem]{Corollary}

\theoremstyle{definition}
\newtheorem{remark}[theorem]{Remark}

\newcommand{\T}{\mathbb{T}}

\newcommand{\F}{\mathbb{F}}
\DeclareMathOperator{\lk}{lk}
\DeclareMathOperator{\st}{st}

\title{On free components of Artin and Coxeter groups}

\author[G. Dumas]{Guillaume Dumas}
\address{\parbox{\linewidth}{Department of Mathematics, University of Maryland, College Park, \\
		4176 Campus Drive, College Park, MD 20742}}
\email{gdumas@umd.edu}
\urladdr{https://perso.ens-lyon.fr/guillaume.dumas/}

\author[J. Huang]{Jingyin Huang}
\address{\parbox{\linewidth}{Department of Mathematics, The Ohio State University,\\
231 W. 18th
Ave, Columbus, OH 43210}}
\email{huang.929@osu.edu}

\author[S. Kunnawalkam Elayavalli]{Srivatsav Kunnawalkam Elayavalli}
\address{\parbox{\linewidth}{Department of Mathematics, University of Maryland, College Park, \\
		4176 Campus Drive, College Park, MD 20742}}
\email{sriva@umd.edu}
\urladdr{https://sites.google.com/view/srivatsavke}

\author[L. Teryoshin]{Lizzy Teryoshin}
\address{\parbox{\linewidth}{Department of Mathematics, University of California, San Diego, \\
		9500 Gilman Drive, La Jolla, CA 92093}}
\email{eteryoshin@ucsd.edu}

\begin{document}

\begin{abstract} The number of connected components can be remembered by the von Neumann algebra among Artin groups, the only possible exception being the case that corresponds to the free group factor problem. In the case of Coxeter groups, this result is obtained in the absence of relatively hyperbolicity. We also discuss a specific case of the analogous problem in measure equivalence where each factor group is a product of nonabelian free groups. 
\end{abstract}
	\maketitle

\section{Introduction}

In recent years, the study of classification and rigidity of von Neumann algebras of graph products has generated some interesting mathematics. We point the reader to \cite{horbez2025rigiditygraphproductvon} and the references therein for a detailed summary of the story. Despite this, there is not much study at the moment into the von Neumann algebras of the family of general Artin or Coxeter groups over defining graphs with arbitrary edge relations (see \cite{blufstein2025strongsolidityclassificationcoxeter} for a recent result on biexactness for general Coxeter groups). 

A rather basic question is whether the number of connected components of the defining graph can be recovered by the group von Neumann algebra, among Artin/Coxeter groups in general. This is a \emph{notorious} question in full generality because it would solve the free group factor problem if proved. Indeed, the connected components could be complete graphs and all of the edge relations could be commutation, in which case the associated von Neumann algebras are amenable, and thus yield free group factors of the appropriate rank \cite{Dykema}. In this note, we show that this is the only obstruction for this phenomenon in Artin groups. 

\begin{theorem}
\label{thm:Artin}
    Suppose $\Gamma_1$ and $\Gamma_2$ are finite simple graphs, and $A_{\Gamma_i}$ are the associated Artin groups with arbitrary choices of defining relations. Suppose each connected component of $\Gamma_1$ is not a complete graph with all edge relations equal to $2$. If $L(A_{\Gamma_1})\cong L(A_{\Gamma_2})$, then $\Gamma_1$ and $\Gamma_2$ have the same number of connected components. 
\end{theorem}

The situation for Coxeter groups is quite a bit more subtle and one cannot hope for the same result as above. This is because fundamental groups of hyperbolic surfaces can appear as free factors in which case it is annoying because it is another \emph{notorious} open problem of Voiculescu whether such groups have the same von Neumann algebra as the free groups. We can, however, prove the result when the assumption of non--relative hyperbolicity is added, which seems to geometrically be almost as good as one can get, owing to the above problem.

\begin{theorem}
\label{thm:Coxeter}
    Suppose $\Gamma_1$ and $\Gamma_2$ are finite simple graphs, and $W_{\Gamma_i}$ are the associated Coxeter groups with arbitrary choices of defining relations. Suppose in addition that each free factor of $W_{\Gamma_i}$ is not relatively hyperbolic and nonamenable. If $L(W_{\Gamma_1})\cong L(W_{\Gamma_2})$, then $\Gamma_1$ and $\Gamma_2$ have the same number of connected components. 
\end{theorem}

We suspect though that the above result should hold more generally when the nonamenable factors admit vanishing of the first $\ell^2-$Betti number, although we are not able to prove this at the moment. Note that the first $\ell^2-$Betti number does vanish for non--relatively hyperbolic Coxeter groups over connected graphs (see Remark \ref{first l2 betti rmk}). 

We point out additionally that the Artin groups and Coxeter groups appearing in Theorem~\ref{thm:Artin} and Theorem~\ref{thm:Coxeter} are precisely those that are thick in the sense of \cite{behrstock2009thick}.

It is interesting to investigate when Artin/Coxeter groups are measure equivalent, in our context of admitting many non trivial free components. While this is a much harder problem, we do obtain below a definitive result in case when the factors are specific types of RAAGs, namely, products of nonabelian free groups. The situation is quite a bit more subtle than the von Neumann algebra case, because it is possible to have measure equivalence, in fact even commensurability, among different numbers of free components. For instance, $(\mathbb{F}_2\times \mathbb{F}_2) \ast (\mathbb{F}_2\times \mathbb{F}_2)$ is in fact commensurable to $(\mathbb{F}_3\times \mathbb{F}_2)*(\mathbb{F}_2\times \mathbb{F}_2)*(\mathbb{F}_2\times \mathbb{F}_2)$, while clearly their von Neumann algebras are not isomorphic. However, as it is shown below, this turns out to be the only way to achieve measure equivalence in this setting. 

\begin{theorem}
     Let $G_1,G_2$ be two groups of the form $\overset{n}{\underset{i=1}{\ast}} (\F_{a_{i,1}}\times\cdots \times \F_{a_{i,k_i}})$. Assume that all of the integers $k_i,a_{i,j}$ are $\geq 2$. Then $G_1$ and $G_2$ are measure equivalent if and only if they are commensurable if and only if their $\ell^2-$Betti numbers are proportional.
\end{theorem}

It is quite easy to compute the $\ell^2-$Betti numbers of the above groups, thanks to the Kunneth formula and additivity under free products. Note that by a result of Gaboriau \cite{Ga02bis}, if two groups are measure equivalent, then their $\ell^2-$Betti numbers are proportional. The above class of examples was considered in the work of Drimbe and Vaes \cite{DriVa25}, in the framework of von Neumann equivalence \cite{IsPeRu19}. We point out that proportionality of $\ell^2-$Betti numbers of von Neumann equivalent groups remains open.  It is a very interesting question whether one could prove that such groups are von Neumann equivalent if and only if they are commensurable. We could not settle this problem at the moment.

 \subsection*{Acknowledgements} We thank University of Maryland and the Ohio State University for hosting visits of the authors, which lead to the completion of this work. We thank C. Ding, A. Ioana, G. Patchell and L. Robert for helpful conversations. 

\section{Some Notation}
Let $\Gamma$ be a finite simple graph with vertex set $V\Gamma$ and edge set $E\Gamma$, with each of its edges labeled by an integer $\ge 2$. A subgraph $\Gamma'$ of a graph $\Gamma$ is \emph{full}, if whenever two vertices of $\Gamma'$ are adjacent in $\Gamma$, then any edge of $\Gamma$ between these two vertices are contained in $\Gamma'$. 

Given a vertex $v\in V\Gamma$, the \textit{link} of $v$, denoted $\lk(v)$, is the set of vertices in $\Gamma$ that are adjacent to $v$. The \textit{star centered at $v$}, denoted $\st(v)$ is defined as $\st(v)=\{v\}\cup\lk(v)$ and the \emph{closed star} of $v$, denoted $N_\Gamma(v)$, is the full subgraph of $\Gamma$ spanned by $v$ and all of the vertices that are adjacent to $v$. We say that a vertex $v\in V\Gamma$ is \textit{essential} if the closed star of $v$ is not properly contained in the closed star of $w$ for a vertex $w\neq v$. A star is \textit{maximal} if it is centered at an essential vertex. Every vertex $v\in V\Gamma$ is contained in some maximal star. 

Given two paths $P_1$, $P_2$ of $\Gamma$ such that $P_1$ ends at the same vertex that $P_2$ starts, we denote by $P_1;P_2$ the concatenation of $P_1$ and $P_2$. Given a set of vertices $V\subset V\Gamma$, we denote by $\Gamma_V$ the full subgraph of $\Gamma$ spanned by $V$ and by $H_V$ the subgroup of $A_\Gamma$ generated by the vertex groups in $V$.

Given two letters $a,b$ and positive integer $m\ge 2$, let $W(a,b,m)$ be the alternating word of length $m$ starting with $a$. 
The \emph{Artin group with defining graph $\Gamma$}, denoted $A_{\Gamma}$, is given by the following presentation:
\begin{center}
	$\langle s_i\in V\Gamma\ |\ W(s_i,s_j,m_{ij})=W(s_j,s_i,m_{ij})$ for each $s_i$ and $s_j$ spanning an edge labeled by $m_{ij}\rangle$.
\end{center}
The \emph{Coxeter group with defining graph $\Gamma$}, denoted $W_{\Gamma}$, is given by the following presentation:
\begin{center}
	$\langle s_i\in V\Gamma\ |\ s_i^2=e, W(s_i,s_j,m_{ij})=W(s_j,s_i,m_{ij})$ for each $s_i$ and $s_j$ spanning an edge labeled by $m_{ij}\rangle$
\end{center}
An Artin or Coxeter group is \emph{right angled} if $W(s_i,s_j,m_{ij})=2$ for all $i,j$. 

Suppose $W$ is a Coxeter group with generating set $S=\{s_1,...,s_n\}$. The pair $(W,S)$ is called a \emph{Coxeter system}. Given a subset $A\subset S$, we denote by $W_A$ the subgroup of $W$ generated by $A$. The Coxeter system $(W_A,A)$ \emph{irreducible} if there is no partition of $A$ into two nonempty, disjoint commuting subsets, \emph{spherical} if $W_A$ is finite, and \emph{minimal non-spherical} if $(W_A,A)$ is non-spherical but $(W_B,B)$ is spherical for every proper subset $B\subset A$.
We say that a subset $A\subset S$ is \emph{irreducible} (resp. \emph{spherical}, \emph{minimal non-spherical}) if $(W_A,A)$ is \emph{irreducible} (resp. \emph{spherical}, \emph{minimal non-spherical}). 
Given $s\in S$, we denote by $s^\perp$ the set of elements in $S\setminus\{s\}$ commuting with $s$, and given two disjoint, commuting subsets $A,B\subset S$, we denote by $(W_A,A)\times(W_B,B)$ the Coxeter system $(W_{A\cup B}, A\cup B)$. Every Coxeter system $(W,S)$ admits a natural decomposition into its \emph{irreducible components}, by which we mean $(W,S)=(W_1,S_1)\times...\times(W_n,S_n)$ where the $W_i$ are irreducible.

We say a Coxeter group is \emph{affine}, if it admits a faithful proper and cocompact isometric action on some Euclidean space such that each generator acts as a Euclidean reflection. Irreducible affine Coxeter groups and irreducible spherical Coxeter groups are classified, see \cite{MR1890629}, or \cite[Table 6.1]{Davisbook}. As an immediate consequence of this classification, we have:
\begin{lemma}
\label{lem:minimal}
Each irreducible affine Coxeter group is minimal non-spherical.
\end{lemma}

We record the following which is a consequence of \cite[Proposition 17.2.1]{Davisbook}.
\begin{lemma}
\label{lem:Coxeter amenable}
    Let $(W_S,S)$ be a Coxeter system. Then $W_S$ is amenable if and only if each of its irreducible components is either affine or spherical. 
\end{lemma}

\section{$\alpha_t$-rigid groups}

Here we will describe a family of countable groups which we call $\alpha_t$--rigid. The motivation of this comes from Popa's deformation/rigidity theory \cite{Po07B}, in particular in the context of the malleable deformations associated to a free product. The proof of the von Neumann algebraic side of the result is in fact well known to experts. The result we aim for is known as ``Kurosh-type rigidity'', of which now there are plenty \cite{Oza06,  Houdayer_2016, DingSri, Drimbe1}, etc. The main input in the argument is that the connected components are $\alpha_t$--rigid, i.e \emph{rigid} with respect to the standard mixing malleable deformation arising from free products \cite{IPP08}. Then the proof follows identically the intertwining arguments in Corollary 8.1 in \cite{Drimbe1}. 

Now we document some examples of groups whose von Neumann algebras will be $\alpha_t$--rigid when they embed into free products endowed with the standard mixing s-malleable deformation.

\begin{proposition}\label{prop:rigidity_conditions}
\begin{enumerate}[(i)]
    \item[]
    \item If $G$ is nonamenable and is the direct sum of two infinite groups, then $G$ is $\alpha_t$--rigid \cite[Lemma 19.1.3]{AP}.
    \item If $G$ is nonamenable and $G$ has infinite center, then $G$ is $\alpha_t$--rigid \cite[Lemma 19.1.3]{AP}.
    \item If $H_1, H_2<G$  are such that $H_i$ are $\alpha_t$--rigid, $H_1,H_2$ generate $G$ and $|H_1\cap H_2|=\infty$ then $G$ is $\alpha_t$--rigid \cite[Theorem 1.3]{Pineapple}.
    \item If $H<G$ is $\alpha_t$--rigid, then $\mathcal{N}_{wq}(H)$ is $\alpha_t$--rigid where $\mathcal{N}_{wq}(H)$ is the group generated by the set of elements $g\in G$ such that $gHg^{-1}\cap H$ is infinite \cite[Theorem 1.6]{Pineapple}. 
\end{enumerate}

\end{proposition}

Owing to this, it will be enough to prove that each of the free components of the groups in our theorems are $\alpha_t$--rigid in the above sense, to derive the main results. 

\section{Artin groups}

\subsection{An algebraic decomposition}

First we treat the right angled case as an illuminating example.

\begin{proposition}
    \label{prop:raag decomposition}
    Suppose $A_\Gamma$ is a right angled Artin group whose defining graph is connected and not complete. Then there exists a finite family of subgroups $H_i\subset A_\Gamma$, $i\in I$ such that the following conditions hold: 
    \begin{enumerate}
        \item $H_i$ is nonamenable for each $i\in I$.
        \item $H_i\cap H_{i+1}$ is infinite.
        \item $H_i\cong A\times B$ where $A,B$ are infinite groups.
        \item $A_\Gamma$ is generated by $\{H_i\}_{i\in I}$.
    \end{enumerate}
\end{proposition}

\begin{proof}
    Let $\mathcal{S}$ be the collection of maximal stars in $\Gamma$. Given a maximal star $S\in\mathcal{S}$, we may pick a representative essential vertex $v_S\in S$ such that $S=\st(v_S)$. Let $\mathcal{V}=\{v_S:S\in\mathcal{S}\}$ be the collection of representative vertices and let $\mathcal{H}=\{H_{\st(v)}\}_{v\in \mathcal{V}}$ be the family of subgroups of $A_\Gamma$ generated by maximal stars of $\Gamma$. We show every $H\in\mathcal{H}$ satisfies condition (1). It is sufficient to show that if $v$ is essential, then $N_\Gamma(v)$ is not complete. Indeed, in this case, there exist two vertices in $N_\Gamma(v)$ that are not connected by an edge, and hence generate a non-abelian free group, from which it follows that $H_{\st(v)}$ is nonamenable. Suppose $N_\Gamma(v)$ were complete. Since $\Gamma$ is not complete, we can pick $w\in V\Gamma\setminus\st(v)$. Let $(v=v_1,v_2,...,v_m=w)$ be a path from $v$ to $w$, and let $v_i$ be the last vertex in the path contained in $\st(v)$. Then $\st(v)\cup\{v_{i+1}\}\subset\st(v_i)$ since $N_\Gamma$ is complete, but $v_{i+1}\not\in\st(v)$, so $\st(v)$ is not maximal.

    We now construct a sequence $\{H_i\}_{i=1}^n$ of groups in $\mathcal{H}$ such that $H_i\cap H_{i+1}$ is infinite and $\{H_i\}_{i=1}^n$ generates $A_\Gamma$. Consider a new graph $\Gamma'$ with vertex set $V\Gamma'=\mathcal{V}$. For $v,w\in V\Gamma'$, let $(v,w)$ be an edge of $\Gamma'$ if $\st(v)\cap\st(w)\neq\varnothing$ (where $\st(v)$ and $\st(w)$ are stars in $\Gamma$). We claim $\Gamma'$ is connected. Indeed, let $v,w\in V\Gamma'$ be two distinct vertices, and let $(v=v_1,v_2,...,v_m=w)$ be a path from $v$ to $w$ in $\Gamma$. For each $v_i$, let $w_i$ be an essential vertex such that $\st(v_i)\subset\st(w_i)$. Since $v_i\in\st(v_i)\cap\st(v_{i+1})\subset\st(w_i)\cap\st(w_{i+1})$, $(w_i,w_{i+1})\in E\Gamma'$ for every $1\le i\le m-1$. Then $(v=w_1,w_2,...,w_m=w)$ is a walk in $\Gamma'$ from $v$ to $w$, so $\Gamma'$ is connected.

    Now let $\{v_i\}_{i=1}^k$ be an enumeration of $\mathcal{V}$. For $1\le i\le k-1$, let $P_i=(v_i=w_{i,1},w_{i,2},...,w_{i,m}=v_{i+1})$ be a path from $v_i$ to $v_{i+1}$ in $\Gamma'$, and let $P=P_1;P_2;...;P_k$ be the concatenation of the $P_i$. Suppose $P=(v_1=w_1,w_2,...,w_n=v_n)$ and set $H_i=H_{\st(w_i)}$ where $\st(w_i)$ is the star in $\Gamma$ centered at the $i$th vertex of $P$. Then $H_i\cap H_{i+1}=H_{\st(w_i)\cap\st(w_{i+1})}$ is infinite since $\st(w_i)\cap\st(w_{i+1})\neq\varnothing$. It follows that $\{H_i\}_{i=1}^n$ is a sequence of groups in $\mathcal{H}$ such that $H_i\cap H_{i+1}$ is infinite, so the sequence $\{H_i\}_{i=1}^n$ satisfies condition (2). Since every vertex of $\Gamma$ is contained in the star of some essential vertex and $P$ contains all essential vertices, it follows that $A_\Gamma$ is generated by $\{H_i\}_{i=1}^n$, so $\{H_i\}_{i=1}^n$ satisfies condition (4).
    
    Finally, since $H_{\{w_i\}}$ commutes with $H_{\lk(w_i)}$, $H_i=H_{\st(w_i)}\simeq H_{\{w_i\}}\times H_{\lk(w_i)}$, so $\{H_i\}_{i=1}^n$ satisfies condition (3).
\end{proof}

We now move on to the case of general Artin groups.

\begin{theorem}(\cite{lek})
	\label{thm:lek}
	Let $\Gamma_1$ be a full subgraph of $\Gamma$ with the induced edge labeling. Then the natural homomorphism $A_{\Gamma_1}\to A_\Gamma$ is injective.
\end{theorem}

The following is standard, see e.g. \cite[\S 2]{brady2002two}.
\begin{theorem}
\label{thm:2generator}
Suppose $\Gamma$ is a single edge labeled by $m$. Then
\begin{enumerate}
    \item if $m\ge 3$, then $A_\Gamma$ has a finite index subgroup that splits as a product of $\mathbb Z$ and a non-abelian free group, in particular, $A_\Gamma$ is not amenable;
    \item if $m$ is odd, then the center of $A_\Gamma$ is a $\mathbb Z$-subgroup generated by $W^2(a,b,m)$, where $a,b$ are vertices of $\Gamma$;
    \item if $m$ is even and $m\neq 2$, then the center of $A_\Gamma$ is a $\mathbb Z$-subgroup generated by $W(a,b,m)$. 
\end{enumerate}
\end{theorem}

For each edge $e$ of $\Gamma$ labeled by $m$ with its vertices being $a$ and $b$, we define $\Delta_e=W(a,b,m)$ if $m$ is even, and $\Delta_e=W^2(a,b,m)$ if $m$ is odd.

Let $v\in \Gamma$ be a vertex and $\{e_i\}_{i=1}^n$ the collection of edges containing $v$. Let $H_v$ be the subgroup of $A_\Gamma$ generated by $v$ and $\{\Delta_{e_i}\}_{i=1}^n$. For each edge $e\subset \Gamma$ with label $\ge 3$, let $H_e$ be the subgroup of $A_\Gamma$ generated by vertices of $e$. By Theorem~\ref{thm:lek} and Theorem~\ref{thm:2generator}, $H_e$ is a nonamenable group with an infinite center.

\begin{lemma}
\label{lem:non amenable vertex}
Let $v\in \Gamma$ be a vertex. Suppose the closed star of $v$ has at least two distinct vertices $u$ and $w$, not equal to $v$, such that $u$ and $w$ do not commute. Then $H_v$ is a nonamenable group with an infinite center.
\end{lemma}

\begin{proof}
By definition, the subgroup generated by $v$ is contained in the center of $H_v$, and the subgroup generated by $v$ is isomorphic to $\mathbb Z$ by Theorem~\ref{thm:lek}, so $H_v$ has infinite center. It remains to show $H_v$ is nonamenable. Let $A_1$ be the subgroup of $A_\Gamma$ generated by $u,v,w$. Then $A_1$ is also an Artin group with the obvious relations by Theorem~\ref{thm:lek}. Let $m_{uw}$ be the label of the edge between $u$ and $w$ in $\Gamma$ (if there is no edge between $u$ and $w$, we define $m_{uw}=\infty$). Similarly, we define $m_{uv}$ and $m_{vw}$. If $\frac{1}{m_{uw}}+\frac{1}{m_{vw}}+\frac{1}{m_{vu}}\le 1$, then $A_1$ is a \emph{2-dimensional} Artin group (see \cite[\S 1]{crisp2005automorphisms}). Let $e_1,e_2$ be the edges $vu,vw$ respectively. Then the proof of \cite[Lem 9(ii)]{crisp2005automorphisms} implies that $\Delta_{e_1}$ and $\Delta_{e_2}$ generate a free group of rank 2. Thus $H_v$ is nonamenable. Now we assume $\frac{1}{m_{uw}}+\frac{1}{m_{vw}}+\frac{1}{m_{vu}}> 1$. In particular $m_{uw}<\infty$, hence $u$ and $w$ are joined by an edge. If $m_{uv}=m_{wv}=2$, then the subgroup generated by $u$ and $w$ is contained in $H_v$. As $m_{uw}\ge 3$ (we are assuming $u$ and $w$ do not commute), by Theorem~\ref{thm:lek} and Theorem~\ref{thm:2generator} (1), this subgroup is nonamenable, hence $H_v$ is nonamenable. It remains to consider the case where at least one of $m_{uv}$ and $m_{wv}$ is $\le 3$. Then the only possibilities for $(m_{uv},m_{wv},m_{uw})$ are $(2,3,3),(2,3,4)$ and $(2,3,5)$ (up to permutation), and the corresponding Artin group $A_1$ is a \emph{spherical Artin group}. Assume without loss of generality that $m_{uv}=2$. Then the subgroup generated by $u$ and $\Delta_{e_2}$ is contained in $H_v$. As $A_1$ is a spherical Artin group, by \cite[Thm 1.2]{jankiewicz2022right}, sufficiently high powers of $u$ and $\Delta_{e_2}$ generates a free group of rank 2, so $H_v$ is nonamenable.
\end{proof}

Suppose $v$ is an essential vertex of $\Gamma$ such that $|\st(v)|\ge3$. If the closed star of $v$ is not complete, then $v$ satisfies the assumptions of Lemma \ref{lem:non amenable vertex} since there are at least two vertices in $N_\Gamma(v)$ not connected by an edge, and that therefore do not commute. 

If the closed star of $v$ is complete, then by the same arguments as in the proof of Proposition \ref{prop:raag decomposition}, one can show that $N_\Gamma(v)=\Gamma$. In particular, since $\Gamma$ is not a complete graph with all edges labeled by 2, there must be two vertices $w,u\in\st(v)$ such that the edge $(w,u)$ is labeled by $m\ge3$. If $w=v$, then we may pick some vertex $v'\in\st(v)$ not equal to $v$ or $u$. Since $N_\Gamma(v)$ is complete and $v$ is essential, $\st(v)=\st(v')$. From this it follows that $v'$ is also an essential vertex. Moreover, $\st(v')$ contains two distinct vertices not equal to $v'$, namely $v$ and $u$, that do not commute. Therefore, given a maximal star $\st(v)$ such that $|\st(v)|\ge3$, we may always find an essential vertex $v'\in\st(v)$ such that $\st(v')$ satisfies the assumptions of Lemma \ref{lem:non amenable vertex} and $\st(v')=\st(v)$.

For an Artin group $A_\Gamma$, we define a family $\mathcal H$ of subgroups of $A_\Gamma$ as follows: For every edge $e$ of $\Gamma$ with label $\ge 3$, $H_e\in\mathcal{H}$, and for every maximal star $S$ such that $|S|\ge3$, $H_v\in \Gamma$ where $v\in S$ is chosen such that $\st(v)=S$ and $v$ satisfies the assumptions of Lemma \ref{lem:non amenable vertex}.

\begin{proposition}
\label{prop:artin group decomposition}
Suppose $\Gamma$ is a connected simplicial graph which is not a complete graph with all edges labeled by 2. Then
\begin{enumerate}
    \item $\mathcal H$ generates $A_\Gamma$;
    \item each member of $\mathcal H$ is nonamenable with infinite center;
    \item we can list elements of $\mathcal H$ as $\{H_1,H_2,\ldots,H_n\}$ (repetition of an element of $\mathcal H$ is allowed in the listing) such that $H_i\cap H_{i+1}$ is infinite for $1\le i\le n-1$.
\end{enumerate}
\end{proposition}

\begin{proof}
    To show (1), let $v$ be a vertex of $\Gamma$ and suppose $v$ is contained in $\st(w)$ where $w$ is an essential vertex. If $|\st(w)|=2$, then $\Gamma$ consists of a single edge $e$ labeled by $m\ge3$, and $\Gamma$ is therefore generated by $H_e$. Otherwise, $|\st(w)|\ge3$, and there exists $H_u$ such that $\st(u)=\st(w)$. If $v=u$, then clearly $H_{\{v\}}\subset H_u$. If the edge $e=(v,u)$ is labeled by 2, then $\Delta_e=vu$, so $v=\Delta_eu^{-1}\in H_u$ and $H_{\{v\}}\subset H_u$. Otherwise, $e=(v,u)$ is labeled by $m\ge3$, so $H_{\{v\}}\subset H_e$. Hence, every vertex group is contained in some $H_u$ or $H_e$ in $\mathcal{H}$, so $\mathcal{H}$ generates $A_\Gamma$.

     Condition (2) follows directly from the definition of $\mathcal{H}$ along with Lemma \ref{lem:non amenable vertex} and the preceding remarks.

    To show (3), let $\Gamma'$ and $P=(w_1,...,w_k)$ be defined as in the proof of Proposition \ref{prop:raag decomposition}. If $H_{w_i}\cap H_{w_{i+1}}$ is not infinite, then given $v\in H_{w_i}\cap H_{w_{i+1}}$, either $e_{i}=(v,w_i)$ or $e_{i+1}=(v,w_{i+1})$ is labeled by $m\ge3$ (since otherwise $H_{w_i}\cap H_{w_{i+1}}$ would contain $H_{\{v\}}$). 
    Suppose first that $e_i$ is labeled by $m\ge3$ and $e_{i+1}$ is labeled by 2. Then $H_{w_i}\cap H_{e_i}$ contains the subgroup generated by $\Delta_{e_i}$, and is therefore infinite. Also, $H_{e_i}\cap H_{w_{i+1}}$ contains $H_{\{v\}}$ and is therefore also infinite. Similarly, if $e_{i+1}$ is labeled by $m\ge3$ and $e_{i}$ is labeled by 2, then $H_{w_i}\cap H_{e_{i+1}}$ and $H_{e_{i+1}}\cap H_{w_{i+1}}$ are both infinite. Now suppose $e_i$ and $e_{i+1}$ are both labeled by $m\ge3$. Since $H_{e_i}\cap H_{e_{i+1}}$ contains $H_{\{v\}}$, it is infinite. Therefore $H_{w_i}\cap H_{e_i}$, $H_{e_i}\cap H_{e_{i+1}}$, and $H_{e_{i+1}}\cap H_{w_{i+1}}$ are all infinite. Hence, by adding a subset of $\{H_{e_i},H_{e_{i+1}}\}$ between $H_{w_i}$ and $H_{w_{i+1}}$ whenever $H_{w_i}\cap H_{w_{i+1}}$ is not infinite, we create a list $\{H_i\}_{i=1}^m$ of element of $\mathcal{H}$ that contains all of the $H_v$ and such that $H_i\cap H_{i+1}$ is infinite. 
    
    Finally, suppose there is some $H_e$ not contained in the list $\{H_i\}_{i=1}^m$. Let $v$ be a vertex of $H_e$ and let $\st(w)$ be a maximal star containing $v$. If the edge $(v,w)$ is labeled by 2, then $H_e\cap H_{w}$ is infinite as it contains $H_{\{v\}}$. If the edge $(v,w)$ is labeled by $m\ge3$, then $H_e\cap H_{(v,w)}$ is infinite as it contains $H_{\{v\}}$, and $H_{(v,w)}\cap H_{w}$ is infinite as it contains $H_{\{w\}}$. We know $H_w$ appears in $\{H_i\}_{i=1}^m$ at least once, so we may replace the first occurrence of $H_w$ with $(H_w,H_e,H_w)$ if $(v,w)$ is labeled by 2 or with $(H_w,H_{(v,w)},H_e,H_{(v,w)},H_w)$  otherwise to obtain a new list $\{H_i\}_{i=1}^m$ containing $H_e$ such that $H_i\cap H_{i+1}$ is infinite for all $1\le i\le m-1$. Repeating this with the remaining $H_e$, we obtain a list $\{H_i\}_{i=1}^n$ containing all elements of $\mathcal{H}$ such that $H_i\cap H_{i+1}$ is infinite for all $1\le i\le n-1$.
\end{proof}
Combining Proposition \ref{prop:rigidity_conditions}(ii,iii) with Proposition \ref{prop:artin group decomposition}, we obtain the following:

\begin{corollary}
    Suppose $\Gamma$ is a connected simplicial graph which is not a complete graph with all edges labeled by 2. Then $A_\Gamma$ is $\alpha_t$--rigid.
\end{corollary}

\section{Coxeter Groups}
We may obtain a similar decomposition for nonamenable, not relatively hyperbolic Coxeter groups. For this, we require the class of Coxeter groups $\T$ and $\T_0$ defined in \cite[Section A.1]{BehrstockThickness2017}.

\begin{definition}
    \label{def:T}
    Define the class of Coxeter systems $\T$ inductively as follows:
    \begin{enumerate}
        \item $\T$ contains $\T_0$, the class of all irreducible affine Coxeter systems $(W,S)$ with $|S|\ge3$ along with Coxeter systems of the form $(W,S)=(W_{S_1},S_1)\times(W_{S_2},S_2)$ where $(W_{S_i},S_i)$ is irreducible non-spherical for $i=1,2$.
        \item If $S=S_0\cup\{s\}$ is such that $(W_{S_0},S_0)\in\T$ and $s^\perp\subset S_0$ is non-spherical, then $(W,S)\in\T$.
        \item If $S=S_1\cup S_2$ is such that $(W_{S_1},S_1),(W_{S_2},S_2)\in\T$ and $S_1\cap S_2$ is non-spherical, then $(W,S)\in\T$.
    \end{enumerate}
\end{definition}
The class $\T$ allows us to characterize relatively hyperbolic Coxeter groups:
\begin{proposition}\cite[Corollary A.10]{BehrstockThickness2017}
    \label{prop:rel hyperbolicity condition}
    A Coxeter system $(W,S)$ is not relatively hyperbolic with respect to any family of proper subgroups if and only if $(W,S)\in\T$.
\end{proposition}

\begin{proposition}
    \label{prop:coxeter group decomposition}
    Let $(W,S)$ be a nonamenable, not relatively hyperbolic Coxeter system with connected defining graph. Then $W$ is $\alpha_t$--rigid.
\end{proposition}

\begin{proof}
    Since $(W,S)$ is not relatively hyperbolic, by Proposition \ref{prop:rel hyperbolicity condition}, $(W,S)\in\T$. It follows that $(W,S)$ can be obtained by a finite number of operations of type (2) or (3) in Definition \ref{def:T}. We proceed by induction on the minimum number of operations, $n$, required to obtain $(W,S)$. If $n=0$, then $(W,S)\in\T_0$. Since all irreducible affine Coxeter systems are amenable by Lemma \ref{lem:Coxeter amenable}, $(W,S)$ must be of the form $(W,S)=(W_{S_1},S_1)\times(W_{S_2},S_2)$ where $S_1$ and $S_2$ are non-spherical. Hence, $W$ is the product of two infinite groups, and is therefore $\alpha_t$--rigid by Proposition \ref{prop:rigidity_conditions}(i).
    
    Now suppose $n\ge1$ and $(W,S)$ is obtained by an operation of type (2) in Definition \ref{def:T}. Then $S=S_0\cup\{s\}$ where $(W_{S_0},S_0)\in\T$ and $s^\perp\subset S_0$ is non-spherical. If $(W_{S_0},S_0)$ is nonamenable, then by inductive hypothesis, $(W_{S_0},S_0)$ is $\alpha_t$--rigid. Since $s^\perp$ is infinite, $sW_{S_0}s^{-1}\cap W_{S_0}$ is infinite, so $s\in\mathcal{N}_{wq}(W_{S_0})$. Hence, $(W,S)$ is $\alpha_t$--rigid by Proposition \ref{prop:rigidity_conditions}(iv).
    
    Now suppose $S_0$ is amenable. Partition $S_0$ into irreducible subsets $A_1,...,A_k$, so that $(W_{S_0},S_0)=(W_{A_1},A_1)\times...\times (W_{A_k},A_k)$. This induces a partition of $s^\perp$ into (possibly empty) subsets $A_i'=A_i\cap s^\perp$, so that $(W_{s^\perp},s^\perp)=(W_{A_1'},A_1')\times...\times (W_{A_k'},A_k')$. Since $(W_{s^\perp},s^\perp)$ is non-spherical, we may suppose without loss of generality that $(W_{A_k'},A_k')$ is non-spherical. Since $(W_{A_k},A_k)$ is non-spherical and amenable, it must be affine by Lemma~\ref{lem:Coxeter amenable}, and therefore minimal non-spherical by Lemma~\ref{lem:minimal}. In particular, $(W_{A_k'},A_k')=(W_{A_k},A_k)$. Then $W=(W_{S'},S')\times(W_{A_k},A_k)$ where $S'=S\setminus A_k$. Since $(W_{A_k},A_k)$ is amenable and $(W,S)$ is nonamenable, $(W_{S'},S')$ must be nonamenable, and in particular, $(W_{S'},S')$ must be infinite. Hence, $W$ is the product of two infinite groups, and is therefore $\alpha_t$--rigid by Proposition \ref{prop:rigidity_conditions}(i).

    If $(W,S)$ is obtained by an operation of type (3) in Definition \ref{def:T}, then $S=S_1\cup S_2$ where $(W_{S_1},S_1),(W_{S_2},S_2)\in\T$ and $S_1\cap S_2$ is non-spherical. If both $(W_{S_1},S_1)$ and $(W_{S_2},S_2)$ are nonamenable, then by inductive hypothesis, $(W_{S_1},S_1)$ and $(W_{S_2},S_2)$ are $\alpha_t$--rigid. Hence, $(W,S)$ is generated by two $\alpha_t$--rigid groups with infinite intersection, so $(W,S)$ is $\alpha_t$--rigid by Proposition \ref{prop:rigidity_conditions}(iii).
    
    Suppose $S_1$ is amenable and $S_2$ is not. As before, we have a partition of $S_1$ into irreducible subsets $A_1,...,A_k$, so that $(W_{S_1},S_1)=(W_{A_1},A_1)\times...\times (W_{A_k},A_k)$ and $(W_{S_1\cap S_2},S_1\cap S_2)=(W_{A_1'},A_1')\times...\times (W_{A_k'},A_k')$ where $A_i'=A_i\cap S_2$. Since $S_1\cap S_2$ is non-spherical, as before, we may assume without loss of generality that $(W_{A_k'},A_k')=(W_{A_k},A_k)$. If $s\in W_{A_i}$, then either $s\in W_{A_k}$ or $s$ commutes with $W_{A_k}$. In either case, $sW_{A_k}s^{-1}=W_{A_k}$, so $|sW_{S_2}s^{-1}\cap W_{S_2}|\ge|W_{A_k}|=\infty$. Hence, $S_1\subset\mathcal{N}_{wq}(W_{S_2})$. Since $W_{S_2}$ is $\alpha_t$--rigid by inductive hypothesis, by Proposition \ref{prop:rigidity_conditions}(iv), $(W,S)$ is $\alpha_t$--rigid as well.
    
    Now suppose $S_1$ and $S_2$ are amenable. Then we have a partition of $S_2$ into irreducible subsets $B_1,...,B_l$, so that $(W_{S_2},S_2)=(W_{B_1},B_1)\times...\times (W_{B_l},B_l)$ and $(W_{S_1\cap S_2},S_1\cap S_2)=(W_{B_1'},B_1')\times...\times (W_{B_l'},B_l')$ where $B_i'=B_i\cap S_1$. For $1\le i\le k$, $1\le j\le l$, let $C_{ij}=A_i'\cap B_j'$. Then we may write $W_{S_1\cap S_2}$ as a direct product $W_{S_1\cap S_2}=\prod_{i,j}(W_{C_{ij}},C_{ij})$.
    Since $(W_{S_1\cap S_2},S_1\cap S_2)$ is non-spherical, we may suppose without loss of generality that $(W_{C_{kl}'},C_{kl}')$ is non-spherical, so $(W_{C_{kl}},C_{kl})=(W_{A_k},A_k)=(W_{B_l},B_l)$. Then $W=(W_{S'},S')\times(W_{C_{kl}},C_{kl})$ where $S'=S\setminus C_{kl}$. Since $(W_{C_{kl}},C_{kl})$ is amenable and $W$ is nonamenable, $(W_{S'},S')$ must be infinite. Hence, $W$ is the product of two infinite groups, and is therefore $\alpha_t$--rigid by Proposition \ref{prop:rigidity_conditions}(i).
\end{proof}

\begin{remark}\label{first l2 betti rmk}
    If $(W,S)\in\T_0$, then $W$ is either amenable or the product of two infinite groups, and so has vanishing first $\ell^2-$Betti number.
    Then combining Proposition \ref{prop:rel hyperbolicity condition} with \cite[Proposition 7.2]{PetersonThom2011}, we see that if $(W,S)$ is not relatively hyperbolic, then the first $\ell^2-$Betti number of $W$ is zero.
\end{remark}

\section{On the measure equivalence analogue}
In the case of measure equivalence, the situation is much more complicated. Consider groups of the form 
$$\overset{n}{\underset{i=1}{\ast}} (F_{a_{i,1}}\times\cdots \times F_{a_{i,k_i}}).$$
We would like to determine when two such groups are measure equivalent. 

\begin{theorem}\label{thm:comfreeproduct}
    Let $G_1,G_2$ be two groups as above. Assume that all of the integers $k_i,a_{i,j}$ are $\geq 2$. Then $G_1$ and $G_2$ are commensurable if and only if their $\ell^2-$Betti numbers are proportional.
\end{theorem}

It is well known (\cite{Ga02bis}) that the proportionality of  $\ell^2$-Betti numbers is a necessary condition, and that we even have the implication $$G_1,G_2 \textrm{ commensurable} \Rightarrow G_1\underset{ME}{\sim}G_2 \Rightarrow \exists c, \forall n, b_n^{(2)}(G_1)=cb_n^{(2)}(G_2).$$

In particular, we obtain the following classification under measure equivalence.
\begin{corollary}\label{coro:mefreeprod}
   Two free products as in Theorem \ref{thm:comfreeproduct} are measure equivalent if and only if they are commensurable, if and only if their $\ell^2$-Betti numbers are proportional.

   Furthermore, there exist measure equivalent RAAGs whose defining graphs have different numbers of connected, not necessarily complete, components.
\end{corollary}

Note that the  $\ell^2-$Betti numbers of these groups are not difficult to compute, and their values are actually an important ingredient in the proof.

We say that a free factor $A_i= F_{a_{i,1}}\times\cdots \times F_{a_{i,k_i}}$ has dimension $m$ if $k_i=m$ (this is, $m$ is the maximal dimension of a free abelian subgroup of $A_i$). Then the only factors contributing to the $m$th $\ell^2$-Betti number are the factors of dimension $m$. More precisely, if $$ G=\overset{n}{\underset{i=1}{\ast}} A_i \textrm{ where }A_i= F_{a_{i,1}}\times\cdots \times F_{a_{i,k_i}},$$then  $$b_1^{(2)}(G)=n-1, \quad b_m^{(2)}(G)=\sum_{k_i=m} b_m^{(2)}(A_i)$$and$$ b_{k_i}^{(2)}(A_i)=\prod_{j=1}^{k_i} (a_{i,j}-1).$$

Let us begin with a simpler yet illuminating case.
\begin{proposition}
    Let $a,b,p,q,c,d,s,t\geq 2$ be integers. Let $G_1=(F_a\times F_b)*(F_p\times F_q)$ and $G_2=(F_c\times F_d)*(F_s\times F_t)$. If 
    $$(a-1)(b-1)+(p-1)(q-1)=(c-1)(d-1)+(s-1)(t-1),$$
    then $G_1$ and $G_2$ are commensurable.
\end{proposition}

\begin{proof} 
    Denote by $A_1,A_2$ the free factors of $G_1$ and by $B_1,B_2$ the free factors of $G_2$.
    Recall that by the Nielsen-Schreier theorem (\cite[Théorème 5]{Serre77}), a subgroup of index $k$ of $F_n$ is free of rank $r=1+k(n-1)$. Choose $(x-1)$ to be a common multiple of $(a-1),(p-1),(c-1),(s-1)$ and $(y-1)$ to be a common multiple of $(b-1),(q-1),(d-1),(t-1)$. Let $H=F_x\times F_y$. Then $H$ is a finite-index subgroup of each free factor. Let $m_{A_1}=[A_1:H]=\frac{(x-1)(y-1)}{(a-1)(b-1)}$, and define the indices $m_{A_2},m_{B_1},m_{B_2}$ similarly.

    Let $N$ be a common multiple of all four indices. We will construct isomorphic subgroups of index $N$ in $G_1$ and $G_2$. To do so, consider the action of $A_1$ on $A_1/H\times \{1,\cdots,N/m_{A_1}\}$ that is trivial in the second component. This is a morphism $A_1\to S_N$. Similarly, the action of $A_2$ gives a morphism $A_2\to S_N$. The universal property of free products then gives a morphism $G_1\to S_N$. Up to conjugating $A_2\to S_N$, we may assume that the action of $G_1$ on $N$ points is transitive. Denoting by $H_1$ the stabilizer of a point, we have that $H_1$ is a subgroup of $G_1$ of index $N$. 

    The group $G_1$ is the fundamental group of a graph of groups with two vertices labeled $A_1$ and $A_2$ with one edge between them, labeled with the trivial group. The Bass-Serre covering tree $T$ has vertices $G_1/A_1\sqcup G_1/A_2$ and edge set $G_1$, with edges connecting $gA_1$ to $gA_2$ for $g\in G_1$. Since $G_1$ acts on $T$, so does $H_1$, and hence $H_1$ is isomorphic to the fundamental group of the quotient graph of groups $H_1\backslash T$ (\cite[Théorème 13]{Serre77}).

    Let $\Gamma$ be the bipartite graph with vertex $V=\{1,\cdots, N/m_{A_1}\}\sqcup \{1,\cdots, N/m_{A_2}\}$. Each vertex is an orbit of the action on $N$ points given by $A_1\to S_N$ or an orbit of the action given by $A_2\to S_{N}$. The vertices are labeled by $H$, and for each $1\leq k\leq N$, there is an edge between the orbit of $k$ under $A_1$ and the orbit of $k$ under $A_2$. We claim that $\Gamma \simeq H_1\backslash T$. Indeed, the vertex set of $H_1\backslash T$ is $H_1\backslash G_1/A_1\sqcup H_1\backslash G_1/A_2$. But $H_1\backslash G_1/A_1\simeq A_1\backslash (G_1/H_1)$ is simply the set of orbits under the action of $A_1$, and the stabilizer is isomorphic to $H$, and similarly for the other set.

    The number of edges of $\Gamma$ is $N$, and the number of vertices is $$V_1=\frac{N}{m_{A_1}}+\frac{N}{m_{A_2}}=\frac{N}{(x-1)(y-1)}\left[(a-1)(b-1)+(p-1)(q-1)\right].$$ Let $R_1=N-V_1+1$ be the number of edges not in a spanning tree of $\Gamma$. Then $$H_1\simeq H^{*V_1}*F_{R_1}.$$

    The same construction on $G_2$ yields a subgroup $H_2$ of index $N$ such that $H_2\simeq H^{*V_2}*F_{R_2}$ with $$V_2=\frac{N}{m_{B_1}}+\frac{N}{m_{B_2}}=\frac{N}{(x-1)(y-1)}\left[(c-1)(d-1)+(s-1)(t-1)\right]$$ and $R_2=N-V_2+1$.

    Thus, by assumption on the integers $a,b,p,q,c,d,s$, and $t$, $V_1=V_2$ and $R_1=R_2$, so $H_1\simeq H_2$ and the proof is complete.
\end{proof}

\begin{proof}[Proof of Theorem \ref{thm:comfreeproduct}] 
First, as discussed above, the fact that commensurability implies proportionality of $\ell^2$-Betti numbers is clear, so we only need to prove the other direction.

    We will proceed as in the simpler case: choosing finite index subgroups of each factor, and constructing a finite index subgroup of the free product whose Kurosh decomposition is controlled by the Betti numbers. However, we cannot get away with a single group; we have to choose different groups depending on the dimensions of the factors.

    Let $G_1=\overset{n_1}{\underset{i=1}{\ast}} A_i$ and $G_2=\overset{n_2}{\underset{i=1}{\ast}} B_i$ have proportional $\ell^2-$Betti numbers, where $A_i=(F_{a_{i,1}}\times\cdots \times F_{a_{i,k_i}})$ and $B_i=(F_{b_{i,1}}\times\cdots \times F_{b_{i,\ell_i}})$.

    First, if $n_1=1$, then the first $\ell^2-$Betti number vanishes, so $n_2=1$ as well. In this case, the only non-zero $\ell^2-$Betti number of $G_1$ is $b_{k_1}^{(2)}(G_1)$, with $k_1$ the dimension of the single factor. Thus the single factor of $G_2$ has the same dimension, and it is well-known (again, using the Nielsen-Schreier theorem and taking direct products of finite index subgroups) that $G_1$ and $G_2$ are commensurable.   
    
    Thus, we may now assume $n_1,n_2\geq 2$. Let $x-1$ be a common multiple of all of the $a_{i,j}-1$. For any $m$, set $X_m=(F_x)^m$. Then if $A_i$ has dimension $m$, $X_m$ is a subgroup of $A_i$ of index $h_{A_i}=\frac{(x-1)^m}{b_m^{(2)}(A_i)}$, and similarly, if $B_i$ has dimension $m$, then $[B_i:X_m]=h_{B_i}=\frac{(x-1)^m}{b_m^{(2)}(B_i)}$. Let $N_1$ be a common multiple of the $h_{A_i}$ and $N_2$ be a common multiple of the $h_{B_i}$. By taking a larger common multiple, we may assume that $N_1(n_1-1)=N_2(n_2-1)$.

    Let us construct a subgroup $H_1$ of $G_1$ of index $N_1$ as follows: for any $i$, $A_i$ acts on $A_i/X_m\times \{1,\cdots,N_1/h_{A_i}\}$, where $m$ is the dimension of $A_i$ and the action is trivial on the second component. This action on $N_1$ points gives a morphism $A_i\to S_{N_1}$. The universal property of free products then gives a morphism $G_1\to S_{N_1}$. We may assume, up to choosing a correct labeling of the $N_1$ points (i.e. by conjugating one of the images of the factors), that the action of $G_1$ is transitive. Thus, we can set $H_1$ to be the stabilizer of a given point, and this is a subgroup of $G_1$ of index $N_1$.

    The group $G_1$ is the fundamental group of a star centered at a vertex $\{e\}$ with trivial edge groups and with $n_1$ other vertices with groups $A_i$. If $T$ is the Bass-Serre universal covering tree of this graph, then the quotient graph of groups $\Gamma=H_1\backslash T$ is as follows: first, there are $N_1$ vertices with the trivial group (corresponding to lifts of the central vertex), each corresponding to one of the $N_1$ points on which the factors $A_i$ act. Then, for each $A_i$, there are $N_1/h_{A_i}$ vertices with group $X_m$, each corresponding to an orbit of the action of $A_i$ on $N_1$ points. Furthermore, for each point $1\leq k\leq N$ and each $1\leq i\leq n_1$, there is an edge with trivial edge group between the central vertex corresponding to $k$ and the vertex of type $A_i$ corresponding to the orbit containing $k$. Then $H_1$ is the fundamental group of $\Gamma$, so 
    $$H_1\simeq \underset{m}{\ast}(X_n)^{\ast V^{(m)}_1}\ast F_{R_1}$$ 
    where $V^{(m)}_1$ is the number of vertices with vertex group $X_m$ and $R_1$ is the number of edges left after removing a spanning tree. Therefore, 
    $$V^{(m)}_1=\sum_{k_i=m} \frac{N_1}{h_{A_i}}=\sum_{k_i=m} \frac{N_1}{(x-1)^m}b_m^{(2)}(A_i)=\frac{N_1}{(x-1)^m}b_m^{(2)}(G_1)$$
    and 
    $$R_1=N_1(n_1-1)-\sum_m V^{(m)}_1+1. $$

    Doing the same construction for $G_2$, we obtain a subgroup $H_2$ of index $N_2$ in $G_2$ such that $$H_2\simeq \underset{m}{\ast}(X_n)^{\ast V^{(m)}_2}\ast F_{R_2}$$with  $$V^{(m)}_2=\sum_{\ell_i=m} \frac{N_2}{h_{B_i}}=\sum_{k_i=m} \frac{N_2}{(x-1)^m}b_m^{(2)}(B_i)=\frac{N_2}{(x-1)^m}b_m^{(2)}(G_2)$$and $$R_2=N_2(n_2-1)-\sum_m V^{(m)}_2+1. $$
    But by assumption, $N_1(n_1-1)=N_2(n_2-1)$. Since $b_1^{(2)}(G_1)=n_1-1$, $b_1^{(2)}(G_2)=n_2-1$, and the $\ell^2-$Betti numbers of $G_1$ and $G_2$ are proportional, we conclude that $H_1\simeq H_2$ and so $G_1,G_2$ are commensurable.
\end{proof} 

\begin{remark}
    As pointed out in Corollary \ref{coro:mefreeprod}, we obtain a complete classification of this family under measure equivalence. If two such groups are ME, then they are von Neumann equivalent in the sense of \cite{IsPeRu19}. It is therefore interesting to ask the following question: if two such groups are von Neumann equivalent, are their $\ell^2-$Betti numbers proportional? (equivalently, are they commensurable?). This would be an interesting generalization of \cite[Thm. 6.4]{DriVa25}. Note that it is still an open question whether the ratio of $\ell^2-$Betti numbers is a von Neumann equivalence invariant.
\end{remark}
    
\bibliographystyle{alpha}
	\bibliography{references}	

\end{document}